\documentclass[12pt]{amsart}
\usepackage[margin=1in]{geometry}
\usepackage{amsmath}
\usepackage{amsfonts}
\usepackage{amssymb}
\usepackage[dvips]{graphics}
\usepackage{epsfig}
\usepackage[latin1]{inputenc}
\usepackage{fancyhdr}
\usepackage{hyperref}
\usepackage{stmaryrd}
\usepackage{epsfig}
\usepackage{tikz}
\usepackage{euscript}
\usepackage{color}
\usepackage{epstopdf}
\usepackage{pgfplots}
\usetikzlibrary{calc}
\usepackage[scaled=1.5]{helvet}		
\usepackage{courier}			
\usepackage{bbm}
\usepackage[latin1]{inputenc}
\usepackage{listings}
\usepackage{graphicx}
\usepackage{amsmath}
\usepackage{amssymb}
\usepackage{bbm}
\newtheorem{theorem}{Theorem}

\newtheorem{definition}[theorem]{Definition}
\newtheorem{example}[theorem]{Example}
\newtheorem{counterexample}[theorem]{Counterexample}

\newtheorem{lemma}[theorem]{Lemma}

\newtheorem{proposition}[theorem]{Proposition}
\newtheorem{remark}[theorem]{Remark}

\usepackage{listings}
\usepackage [all,arc,curve,color, arrow, matrix,frame]{xy}
\usepackage[english]{babel}
\usepackage{tikz}

\newcommand{\R}{{\mathbb R}}
\newcommand{\C}{{\mathbb C}}

\newcommand{\Real}{{\mbox{Re}}}

\newcommand{\spann}{{\mbox{span}}}
\newcommand{\ran}{{\mbox{ran}}}
\newcommand{\Imag}{{\mbox{Im}}}
\newcommand{\idty}{{\mathbbm{1}}}

\numberwithin{equation}{section}
\numberwithin{theorem}{section} 
\numberwithin{footnote}{section}
\begin{document}

\title[Dissipative extensions]{Extensions of
	dissipative operators with closable imaginary part}

\author[C. Fischbacher]{Christoph Fischbacher}
\address{Department of Mathematics\\
	University of California, Irvine\\
	Irvine, CA, 92697, USA}
\email{fischbac@uci.edu}

\date{}

\maketitle

%
\begin{abstract} Given a dissipative operator $A$ on a complex Hilbert space $\mathcal{H}$ such that the quadratic form $f\mapsto \Imag\langle f,Af\rangle$ is closable, we give a necessary and sufficient condition for an extension of $A$ to still be dissipative. As applications, we describe all maximally accretive extensions of strictly positive symmetric operators and all maximally dissipative extensions of a highly singular first-order operator on the interval.
\end{abstract}
\section{Introduction}

The purpose of this paper is to contribute towards the extension theory of dissipative operators $A$ on complex Hilbert spaces $\mathcal{H}$, for which the quadratic form $q_A:\:f\mapsto \Imag\langle f,Af\rangle$ is closable, which we use to define an ``imaginary part" $V_A$ of $A$. 

The extension theory of symmetric, sectorial, accretive and dissipative operators is an extensively studied subject and it would go beyond the scope of this paper to give a complete overview over this subject (for more background information, we recommend the surveys \cite{Arli,AT2009} as well as the monograph \cite{Behrndtetal} and the references therein).

Since finding the dissipative extensions of a dissipative operator is equivalent to finding the contractive extensions of a given contraction (via Cayley transforms, cf.\ \cite[Thm.\ 1.1.1]{Phillips}), this problem  has in principle been completely solved by Crandall \cite[Thm.\ I and Cor.\ I]{CrandallContraction}. Given a contraction defined on a closed subspace $\mathcal{C}\subset\mathcal{H}$, Crandall gave a description of all its contractive extensions on the whole Hilbert space $\mathcal{H}$ (cf.\ also the results in  \cite{AG82}). However, for practical applications, it can be very difficult to explicitly compute the operators involved.
 
An approach which therefore can be taken is to make additional assumptions on the structure of the dissipative operators under consideration. For instance, in \cite{Crandall}, the dissipative extensions of operators $A$ of the form $A=S+iV$, where $S$ is symmetric and $V\geq 0$ is selfadjoint, were studied.
Similarly, in a series of previous papers \cite{FNW, Nonproper, BKVG}, we studied extension problems of dissipative and sectorial operators $A$ of the form $A=S+iV$, where $S$ and $V$ are symmetric.
Hence, requiring that $A$ is a dissipative operator such that $q_A$ is closable is more general, since it is always satisfied by operators of the form $A=S+iV$.

We will proceed as follows:

In Section \ref{sec:2}, we review some basic definitions and previous results concerning dissipative operators and closable quadratic forms. Moreover, given a dissipative operator $A$ such that the quadratic form $q_A: f\mapsto \Imag\langle f,Af\rangle$ is closable, we introduce the ``imaginary part" $V_A$ of $A$ to be the non--negative selfadjoint operator associated to the closure of $q_A$ (Def.\ \ref{def:quadform}). By construction, it follows that $\mathcal{D}(A)\subset\mathcal{D}(V_A^{1/2})$, but we show in Counterexample \ref{counter:AsubV} that it is in general not true that $\mathcal{D}(A)\subset\mathcal{D}(V_A)$.

We then prove our main theorem (Thm.\ \ref{thm:operator}) in Section \ref{sec:3}. Provided that the imaginary part $V_A$ is strictly positive, we give a necessary and sufficient condition for an extension $B$ of $A$ to be dissipative, which we only have to check for the elements of an arbitrary fixed subspace $\mathcal{V}$ that is complementary to $\mathcal{D}(A)$ in $\mathcal{D}(B)$. While our main result involves operators that generally can be difficult to compute, we discuss two applications where everything can be done explicitly and believe that our result can be of use for future applications.

The remainder of this paper is dedicated to discussing these two applications of Theorem \ref{thm:operator}. In Section \ref{sec:4}, we consider the problem of finding the accretive extensions of a given strictly positive symmetric operator $S$ (Thm.\ \ref{thm:accretive}). In particular, we will show that for an extension of $S$ to be accretive, it is necessary that its domain be contained in $\mathcal{D}(S_K^{1/2})$ -- the form domain of the Kre{\u\i}n-von Neumann extension of $S$ (Lemma \ref{lemma:wadjoint}). As an example, we discuss the maximally accretive extensions of the Schr\"odinger operator $-\frac{\text{d}^2}{\text{d}x^2}+V$ on $L^2(\R^+)$, where we assume $V\in L^\infty(\R^+)$ and $V\geq \varepsilon>0$ almost everywhere (Example \ref{ex:schr}).

In Section \ref{sec:5}, we discuss the dissipative extensions of the highly singular first-order operator $i\frac{\text{d}}{\text{d}x}+i\frac{\gamma}{x}
$ on $L^2(0,1)$, where $\gamma>0$. It turns out that the concrete structure of this operator makes it possible to compute all conditions given in Theorem \ref{thm:operator} explicitly and therefore give a full description of all its maximally dissipative extensions (Prop.\ \ref{prop:firstordermaxdiss}).

\section{Preliminaries} \label{sec:2}
To begin with, let us introduce the following conventions: all inner products considered here are linear in the second component, i.e.\ for any $f,g\in\mathcal{H}$ and any $\lambda\in\C$, we have $\langle g,\lambda f\rangle=\lambda\langle g,f\rangle=\langle \overline{\lambda}g,f\rangle$. Moreover, for an operator $A$ on a Hilbert space $\mathcal{H}$, let $\mathcal{D}(A), \ran(A), \ker(A)$ and $\overline{A}$ denote its domain, range, kernel and closure, respectively. Lastly, for any subspace $\mathcal{D}\subset\mathcal{D}(A)$, the restriction of $A$ to $\mathcal{D}$ is denoted by $A\upharpoonright_\mathcal{D}$.

Let us now give a few basic definitions:
\begin{definition}[Dissipative operators] A densely defined operator $A$ on a Hilbert space $\mathcal{H}$ is called {\bf{dissipative}} if and only if 
\begin{equation}
\mbox{\emph{Im}}\langle f,Af\rangle\geq 0
\end{equation}
for any $f\in\mathcal{D}(A)$. Moreover, if $A$ is dissipative and has no non-trivial dissipative operator extension, then it is called {\bf{maximally dissipative}}.
\label{def:dissipative} 
\end{definition}
\begin{definition}[{Closable quadratic form, cf.\ \cite[VI, \S 1, Sec.\ 4]{Kato}}]
Let $q$ be a quadratic form. Then, $q$ is called {\bf closable} if and only if for any sequence $\{f_n\}_n\subset\mathcal{D}(q)$, we have that if
\begin{equation} \|f_n\|\overset{n\rightarrow\infty}{\longrightarrow}0\quad\text{and}\quad q(f_n-f_m)\overset{n,m\rightarrow\infty}{\longrightarrow}0\:,
\end{equation}
then this implies that 
\begin{equation}
q(f_n)\overset{n\rightarrow\infty}{\longrightarrow}0\:.
\end{equation}
 \label{def:closable}
If $q$ is closable, its {\bf closure} $q'$ is given by \cite[VI, Thm.\ 1.17]{Kato}:
\begin{align}
q': \quad&\mathcal{D}(q')=\{ f\in\mathcal{H}: \exists \{f_n\}_{n}\subset\mathcal{D}(q)\:\: s.t.\:\: \|f_n-f\|\overset{n\rightarrow\infty}{\longrightarrow}0\:\:\text{and}\:\: q(f_n-f_m)\overset{n,m\rightarrow\infty}{\longrightarrow} 0\}\\
&\:q'(f):=\lim_{n\rightarrow\infty}q(f_n)\notag\:.
\end{align}
\end{definition}
Let us also recall the following useful Lemma:
\begin{lemma}[Mentioned in \cite{Crandall}, see also \cite{thesis} for a proof.] \label{prop:dissdomain} 
Let $A$ be a closed and dissipative operator on a separable Hilbert space $\mathcal{H}$ such that $\dim\text{\emph{ker}}(A^*-i)<\infty$. Moreover, let $\widehat{A}$ be a dissipative extension of $A$. Then, $\widehat{A}$ is maximally dissipative if and only if 
\begin{equation}
\dim \mathcal{D}(\widehat{A})/{\mathcal{D}(A)}=\dim\text{\emph{ker}}(A^*-i)\:.
\end{equation}
\end{lemma} 
If a dissipative operator $A$ is such that $q_A$ is closable, let us now define the imaginary part associated to $A$.
\begin{definition} Let $A$ be a dissipative operator. We then define the quadratic form $q_A$ as follows:
\begin{align}
q_A:\quad\mathcal{D}(q_A)=\mathcal{D}(A),\quad
f \mapsto \mbox{\emph{Im}}\langle f,Af\rangle\:.
\end{align}
If $q_A$ is closable, let $V_A$ denote the non-negative selfadjoint operator associated to $q_A'$. \label{def:quadform}
\end{definition}
There are of course well-known examples of non-closable quadratic forms, which we use in the following to construct a dissipative operator which has a non-closable imaginary part: 
\begin{counterexample}[Momentum operator on the half-line] \normalfont Let $\mathcal{H}= L^2(\R^+)$ and consider the operator
\begin{align}
A:\quad\mathcal{D}(A)=H^1(\R^+), \quad f\mapsto -if'\:,
\end{align}
where here and in the following $f'$ denotes the weak derivative of $f$.
For any $f\in\mathcal{D}(A)$, we have --- using integration by parts ---
\begin{align}
\Imag\langle f,Af\rangle=\Imag\langle f,-if'\rangle =\frac{|f(0)|^2}{2}\geq 0\:,
\end{align}
which means that $A$ is dissipative. But the form given by 
\begin{equation}
q_A(f):=\Imag \langle f,Af\rangle=\frac{|f(0)|^2}{2}
\end{equation}
with domain $\mathcal{D}(q_A)=H^1(\R^+)$ is easily seen to not be closable. 
\end{counterexample}
\begin{remark} In previous papers \cite{FNW, Nonproper, BKVG}, we studied the dissipative extensions of dissipative operators $A$ that can be written in the form $A=S+iV$, where $S$ and $V\geq 0$ are symmetric and $\mathcal{D}(A)=\mathcal{D}(S)=\mathcal{D}(V)$ are dense. Note that in this special case, the quadratic form $q_A(f)=\langle f,Vf\rangle$ is always closable \cite[Thm.\ VI, 1.27]{Kato}. 
\end{remark}
Even though by construction we have $\mathcal{D}(q_A)=\mathcal{D}(A)\subset\mathcal{D}(V_A^{1/2})=\mathcal{D}(q_A')$, it is not in general true that
$\mathcal{D}(A)\subset\mathcal{D}(V_A)$ as the following counterexample shows:
\begin{counterexample} \normalfont Let $b$ be such that $\Real(b)\geq 0$ and $\Imag(b)\neq 0$ and consider the maximally dissipative operator $A$ on $\mathcal{H}=L^2(\R^+)$ given by
\begin{align}
A:\qquad\mathcal{D}(A)=\{f\in H^2(\R^+): f'(0)=bf(0)\},\quad 
f\mapsto -if''+if\:.
\end{align}
The quadratic form ${q}_A$ induced by the imaginary part of $A$ is given by
\begin{align}
{q}_A(f)=\Imag\langle f,Af\rangle&=\Imag \left(\int_0^\infty \overline{f(x)}\left(-if''(x)\right)\text{d}x\right)+\|f\|^2\\&=\Real(b)\cdot|f(0)|^2+\|f'\|^2+\|f\|^2\geq 0 \notag
\end{align}
and since
\begin{equation} \label{eq:estimate}
\left||f(0)|^2\right|=\left| \int_0^\infty 2\Real (\overline{f(x)}f'(x))\text{d}x\right|\leq 2\int_0^\infty |f(x)||f'(x)|\text{d}x\leq \|f\|^2+\|f'\|^2
\end{equation}
we have
\begin{equation}
\|f\|^2+\|f'\|^2\leq {q}_A(f) \overset{\eqref{eq:estimate}}{\leq}(1+\Real(b))\left(\|f\|^2+\|f'\|^2\right)\:,
\end{equation}
which means that the norm induced by ${q}_A$ is equivalent to the first Sobolev norm, which implies in particular that $q_A$ is closable. Closing $\mathcal{D}(A)$ with respect to this norm just yields the first Sobolev space $H^1(\R^+)$ and the selfadjoint operator $V_A$ associated to the closed form $q'_A$ is given by
\begin{align}
V_A:\qquad\mathcal{D}(V_A)=\{f\in H^2(\R^+): f'(0)=\Real (b)f(0)\}, \quad f\mapsto -f''+f\:.
\end{align}
Hence, since $b\neq \Real(b)$, we have constructed an example where $\mathcal{D}(A)\not\subset\mathcal{D}(V_A)$.
\label{counter:AsubV}
\end{counterexample}
\section{The main theorem} \label{sec:3}

Let $A$ be dissipative and assume that it induces a strictly positive closable imaginary part $q_A$. In the following, we will determine a necessary and sufficient condition for an extension $A\subset B$ to be dissipative.
\begin{theorem} Let $A$ be dissipative and let ${q}_A$ be the quadratic form as defined in Definition \ref{def:quadform}. Assume that ${q}_A$ is closable and that there exists an $\varepsilon>0$ such that
\begin{equation} \label{eq:totalpositive}
{q}_A(f)\geq\varepsilon\|f\|^2
\end{equation}
for all $f\in\mathcal{D}(A)=\mathcal{D}(q_A)$. Let $V
_A$ be the selfadjoint operator associated to the closure of ${q}_A$. Moreover, let $W_A$ be the operator given by
\begin{align} W_A:\qquad \mathcal{D}(W_A)=\mbox{\emph{ran}}(V_A^{1/2}\upharpoonright_{\mathcal{D}(A)}),\quad
g\mapsto AV_A^{-1/2}g\:.
\end{align}
Let $B$ be an extension of $A$ and pick any subspace $\mathcal{V}$ that is complementary to $\mathcal{D}(A)$ in $\mathcal{D}(B)$, i.e.\ $\mathcal{D}(A)\cap\mathcal{V}=\{0\}$ and $\mathcal{D}(B)=\mathcal{D}(A)\dot{+}\mathcal{V}$. Then, $B$ is dissipative if and only if $\mathcal{V}\subset\mathcal{D}(W_A^*)$ and the inequality
\begin{equation} \label{eq:conditiondiss}
\mbox{\emph{Im}}\langle v,Bv\rangle\geq\frac{1}{4}\|(V_A^{-1/2}B-W_A^*)v\|^2
\end{equation} \label{thm:operator}
is satisfied for every $v\in\mathcal{V}$.
\end{theorem}
\begin{proof} We need to show that $\Imag\langle f+v,B(f+v)\rangle\geq 0$ for any $f\in\mathcal{D}(A)$ and any $v\in\mathcal{V}$. Using that for any $f\in\mathcal{D}(A)$, we get $\Imag\langle f,Af\rangle=\|V_A^{1/2}f\|^2$, we can write
\begin{align}
\Imag\langle f+v,B(f+v)\rangle&=\Imag\langle f,Af\rangle+\Imag\langle v,Bv\rangle+\Imag\langle f,Bv\rangle+\Imag\langle v,Af\rangle\notag\\
&=\|V_A^{1/2}f\|^2+\Imag\langle v,Bv\rangle+\Imag\langle f,Bv\rangle+\Imag\langle v,Af\rangle\:.
\end{align}
Now, since $V_A\geq\varepsilon$, it follows that $V_A^{1/2}$ is boundedly invertible. Hence, for any $f\in\mathcal{D}(A)\subset\mathcal{D}(V_A^{1/2})$, there exists a unique $g\in\mathcal{H}$ such that $f=V_A^{-1/2}g$. Moreover, note that ${\ran(V_A^{1/2}\upharpoonright_{\mathcal{D}(A)})}$ is dense in $\mathcal{H}$. This follows from the fact that $\ran(V_A^{1/2})=\mathcal{H}$ and that for any $V_A^{1/2}f$, where $f\in\mathcal{D}(V_A^{1/2})$, there exists a sequence $\{f_n\}_n\subset\mathcal{D}(A)$ such that $V_A^{1/2}f_n\rightarrow V_A^{1/2}f$ since $\mathcal{D}(A)$ is a core for $V_A^{1/2}$. This means that $W_A$ is a densely defined operator. Now, let us write
\begin{align} \label{eq:formim}
\Imag\langle f+v,B(f+v)\rangle&=\Imag\langle V_A^{-1/2}g+v,B(V_A^{-1/2}g+v)\rangle\notag\\&=\|V_A^{1/2}V_A^{-1/2}g\|^2+\Imag\langle v,Bv\rangle+\Imag\langle V_A^{-1/2}g,Bv\rangle+\Imag\langle v,AV_A^{-1/2}g\rangle\notag\\
&=\|g\|^2+\Imag\langle v,Bv\rangle+\Imag\langle g,V_A^{-1/2}Bv\rangle+\Imag\langle v,W_Ag\rangle\:.
\end{align}
Assume that $v\notin\mathcal{D}(W_A^*)$, which would mean that the map $g\mapsto\langle v,W_Ag\rangle$ is an unbounded linear functional on $\mathcal{D}(W_A)=\ran(V_A^{1/2}\upharpoonright_{\mathcal{D}(A)})$. Hence, there exists a normalized sequence $\{g_n\}_n\subset\mathcal{D}(W_A)$ such that $\Imag\langle v,W_Ag_n\rangle\overset{n\rightarrow\infty}{\longrightarrow}-\infty$. Plugging this into Equation \eqref{eq:formim}, we obtain
\begin{align}
\|g_n\|^2+\Imag\langle v,Bv\rangle&+\Imag\langle g_n,V_A^{-1/2}Bv\rangle+\Imag\langle v,W_Ag_n\rangle\notag\\&\leq 1+\Imag\langle v,Bv\rangle+\|V_A^{-1/2}Bv\|+\Imag\langle v,W_Ag_n\rangle\overset{n\rightarrow\infty}{\longrightarrow}-\infty\:,
\end{align}
which means that $B$ cannot be dissipative in this case. Thus, suppose that for any $v\in\mathcal{V}$, we have $v\in\mathcal{D}(W_A^*)$ from now on. Let us now show that if \eqref{eq:conditiondiss} is satisfied for all such $v$, this implies that $B$ is dissipative. We proceed to estimate \eqref{eq:formim}:
\begin{align}
&\|g\|^2+\Imag\langle v,Bv\rangle+\Imag\langle g,V_A^{-1/2}Bv\rangle+\Imag\langle v,W_Ag\rangle\notag\\
\overset{\eqref{eq:conditiondiss}}{\geq}&\|g\|^2+\frac{1}{4}\|(V_A^{-1/2}B-W_A^*)v\|^2+\Imag\langle W_A^*v,g\rangle-\Imag\langle V_A^{-1/2}Bv,g\rangle\notag\\
\geq&\|g\|^2+\frac{1}{4}\|(V_A^{-1/2}B-W_A^*)v\|^2-\|g\|\|(V_A^{-1/2}B-W_A^*)v\|\notag\\=&\left(\|g\|-\frac{1}{2}\|(V_A^{-1/2}B-W_A^*)v\|\right)^2\geq 0\:,
\end{align}
which shows that \eqref{eq:conditiondiss} is sufficient for $B$ to be dissipative. Let us finish the proof by showing that it is also necessary. To this end, assume that there exists a $v\in\mathcal{V}$ for which we have 
\begin{equation} \label{eq:cond1}
\Imag\langle v,Bv\rangle-\frac{1}{4}\|(V_A^{-1/2}B-W_A^*)v\|^2\leq -\varepsilon
\end{equation}
for some $\varepsilon>0$. Since $\mathcal{D}(W_A)$ is dense, we may pick a sequence $\{g_n\}_n\subset\mathcal{D}(W_A)$ such that $g_n\overset{n\rightarrow\infty}{\longrightarrow} \frac{-i}{2}(V_A^{-1/2}B-W_A^*)v$. Plugging this sequence into \eqref{eq:formim}, we get
\begin{align}
\notag\|g_n\|^2+\Imag\langle v,Bv\rangle+&\Imag\langle (W_A^*-V_A^{-1/2}B)v,g_n\rangle \\\overset{n\rightarrow\infty}{\longrightarrow}\:&\Imag\langle v,Bv\rangle-\frac{1}{4}\|(V_A^{-1/2}B-{W_A}^*)v\|^2\overset{\eqref{eq:cond1}}{\leq}-\varepsilon\:.
\end{align}
This shows that $B$ cannot be dissipative in this case either, which finishes the proof.
\end{proof}
\section{Accretive extensions of strictly positive symmetric operators} \label{sec:4}
In this section, we apply Theorem \ref{thm:operator} in order to determine the accretive extensions of a strictly positive symmetric operator $S$. We start with the following two definitions:
\begin{definition}[Strictly positive symmetric operator]
A symmetric (and in particular densely defined) operator is called {\bf strictly positive} if there exists a $\varepsilon>0$ such that
\begin{equation}
\langle f,Sf\rangle\geq \varepsilon\|f\|^2
\end{equation}
for any $f\in\mathcal{D}(S)$.
\end{definition}
\begin{definition}[Accretive operator]
An operator $A$ on a Hilbert space $\mathcal{H}$ is called {\bf (maximally) accretive} if and only if $(iA)$ is (maximally) dissipative.
\end{definition}
Clearly, the problem of determining accretive extensions of a strictly positive symmetric operator $S$ is equivalent to determining the dissipative extensions of the dissipative operator $A:=iS$, for which we have a result formulated in Theorem \ref{thm:operator}. We will thus only consider the latter problem from now on, while noting that an operator $B$ being a dissipative extension of $A=(iS)$ implies that $(-iB)$ is an accretive extension of $S$.
Before we proceed and prove the theorem, we need to recall the definition of the Kre{\u\i}n-von Neumann extension of a strictly positive symmetric operator $S$ and some of its properties:
\begin{definition}
Let $S$ be a strictly positive symmetric operator. Moreover, let $S_F$ denote the Friedrichs extension of $S$ (for more details about how $S_F$ is constructed, see for example \cite[Chapter 2.2]{Teschl}). We then define the {\bf{Kre{\u\i}n-von Neumann}} extension of $S$, which we denote by $S_K$, as follows
\begin{align}
S_K:\qquad\mathcal{D}(S_K)=\mathcal{D}(\overline{S})\dot{+}\ker(S^*),\quad
S_K=S^*\upharpoonright_{\mathcal{D}(S_K)}\:.
\end{align}
\end{definition}
\begin{proposition} \label{coro:andonishio}
The domain of $S_K^{1/2}$ is given by 
\begin{equation} \label{eq:kreindomain}
\mathcal{D}(S_K^{1/2})=\mathcal{D}(S_F^{1/2})\dot{+}\ker(S^*)
\end{equation}
and for any $f\in\mathcal{D}(S_F^{1/2})$ and any $v\in\ker(S^*)$, one has
\begin{equation} \label{eq:kreinform}
\|S_K^{1/2}(f+v)\|^2=\|S_F^{1/2}f\|^2\:.
\end{equation}
Moreover, $\mathcal{D}(S_K^{1/2})$ can also be characterized as follows 
\begin{align} \label{eq:kreinformdomain}
\mathcal{D}(S_K^{1/2})=\left\{h\in\mathcal{H}: \sup_{g\in\text{\emph{ran}}(S_F^{1/2}\upharpoonright_{\mathcal{D}(S)}):\|g\|=1}{|\langle h,S_F^{1/2}g\rangle|}<\infty\right\}\:.
\end{align}
\end{proposition}
\begin{proof} Formulas \eqref{eq:kreindomain} and \eqref{eq:kreinform} are a well-known result and are for example discussed in \cite{Alonso-Simon}. The characterization \eqref{eq:kreinformdomain} of $\mathcal{D}(S_K^{1/2})$ is shown in \cite[Corollary 2.4]{Nonproper} and is a Corollary of a result by Ando and Nishio {\cite[Thm.\ 1]{Ando-Nishio}}.
\end{proof}
\begin{lemma} \label{lemma:wadjoint} Let $A=iS$, where $S$ is a strictly positive symmetric operator and let $S_F$ denote the strictly positive and selfadjoint Friedrichs extension of $S$. Then, the operator $W_A^*$ is given by
\begin{align} \label{eq:wadjoint}
W_A^*:\qquad\mathcal{D}(W_A^*)=\mathcal{D}(S_K^{1/2})=\mathcal{D}(S_F^{1/2})\dot{+}\ker(S^*),\quad
f\mapsto -iS_F^{1/2}\mathcal{P}f\:, 
\end{align}
where $\mathcal{P}$ is the (unbounded) projection onto $\mathcal{D}(S_F^{1/2})$ along $\ker(S^*)$ according to the decomposition $\mathcal{D}(S_K^{1/2})=\mathcal{D}(S_F^{1/2})\dot{+}\ker(S^*)$.
\end{lemma} 
\begin{proof} Firstly, observe that for $A=iS$, the quadratic form $q_A$ is given by
\begin{align}
q_A:\mathcal{D}(q_A)=\mathcal{D}(A),\quad
f\mapsto\Imag\langle f,iAf\rangle=\langle f,Sf\rangle\:,
\end{align}
which is closable by \cite[Thm.\ VI, 1.27]{Kato} from which we also get that $V_A=S_F$. Thus, the operator $W_A$ is given by
\begin{align}
W_A:\qquad\mathcal{D}(W_A)=\ran(S^{1/2}_F\upharpoonright_{\mathcal{D}(S)}),\quad
f\mapsto iSS_F^{-1/2}f\:.
\end{align}
Let us now show \eqref{eq:wadjoint}. We start by assuming that $v\in\mathcal{D}(S_K^{1/2})=\mathcal{D}(S_F^{1/2})\dot{+}\ker(S^*)$. For any $f\in\mathcal{D}(W_A)=\ran(S_F^{1/2}\upharpoonright_{\mathcal{D}(S)})$, we then get that
\begin{align}
\langle v,W_Af\rangle&=\langle v,iSS_F^{-1/2}f\rangle=\langle \mathcal{P}v,iSS_F^{-1/2}f\rangle+\underbrace{\langle (\idty-\mathcal{P})v,iSS_F^{-1/2}f\rangle}_{=0}\notag\\&=\langle \mathcal{P}v,iS_F ^{1/2}S_F^{1/2}S_F^{-1/2}f\rangle=\langle -iS_F^{1/2}\mathcal{P}v,f\rangle\:,
\end{align}
which shows that $\mathcal{D}(S_K^{1/2})=(\mathcal{D}(S_F^{1/2})\dot{+}\ker(S^*))\subset\mathcal{D}(W_A^*)$ and that $W_A^*v=-iS^{1/2}_F\mathcal{P}v$ for $v\in\mathcal{D}(S_K^{1/2})$.

Let us now show that $\mathcal{D}(W_A^*)\subset \mathcal{D}(S_K^{1/2})$. Assume that this is not true, i.e. that there exists a $v\in\mathcal{D}(W_A^*)$ such that $v\notin\mathcal{D}(S_K^{1/2})$. If $v\in\mathcal{D}(W_A^*)$, this means that there exists a $C<\infty$ such that for any $f\in\mathcal{D}(W_A)=\ran(S^{1/2}_F\upharpoonright_{\mathcal{D}(S)})$ we have
\begin{equation} \label{eq:boundedfunctional}
|\langle v,W_Af\rangle|\leq C\|f\|\:.
\end{equation}
Note that $|\langle v,W_Af\rangle|=|\langle v,SS_F^{-1/2}f\rangle|=|\langle v,S_F^{1/2}f\rangle|$. Since $v\notin\mathcal{D}(S_K^{1/2})$, it now follows from Proposition \ref{coro:andonishio}, that there exists a normalized sequence $\{f_n\}_n\subset\ran(S_F^{1/2}\upharpoonright_{\mathcal{D}(S)})$ such that
$\lim_{n\rightarrow\infty}|\langle v,S_F^{1/2}f_n\rangle|=+\infty$. But this means that \eqref{eq:boundedfunctional} cannot be satisfied in this case, which shows that $\mathcal{D}(W_A^*)\subset\mathcal{D}(S_K^{1/2})$ and thus \eqref{eq:wadjoint}, which finishes the proof.
\end{proof}
The previous lemma shows that any dissipative extension of $A$ has to have domain contained in $\mathcal{D}(S_K^{1/2})$. We will therefore introduce the following characterization of possible dissipative extensions of $A$:
\begin{definition} \label{def:extension}
Let $A=iS$, where $S$ is a strictly positive symmetric operator. Let $\mathcal{V}\subset\mathcal{D}(S_K^{1/2})$ be a subspace such that $\mathcal{V}\cap\mathcal{D}(A)=\{0\}$. Moreover, let $\mathcal{L}$ be an operator from $\mathcal{V}$ to $\mathcal{D}(S_F)$. We then define the operator $A_{\mathcal{V,L}}$ as follows
\begin{align}
A_{\mathcal{V,L}}:\qquad\mathcal{D}(A_{\mathcal{V,L}})=\mathcal{D}(A)\dot{+}\mathcal{V},\quad
f+v\mapsto iSf+iS_F\mathcal{L}v\:,
\end{align}
where $f\in\mathcal{D}(A)$ and $v\in\mathcal{V}$.
\end{definition} 
\begin{lemma} The operators $A_{\mathcal{V,L}}$ are well-defined. Moreover, they characterize all possible extensions of $A$ that have domain contained in $\mathcal{D}(S_K ^{1/2})$. \label{lemma:param}
\end{lemma}
\begin{proof}  Well-definedness follows from the assumptions made on $\mathcal{V}$ and $\mathcal{L}$. Let $B$ with $\mathcal{D}(B)\subset\mathcal{D}(S_K^{1/2})$ be an extension of $A$. This means in particular that there exists a subspace $\mathcal{V}_B$ such that $\mathcal{D}(B)=\mathcal{D}(A)\dot{+}\mathcal{V}_B$. Now, for any $v\in\mathcal{V}_B$, let us define the operator $\mathcal{L}_Bv:=-iS_F^{-1}Bv$. Note that invertibility of $S_F$ follows from the assumption that $S$ is strictly positive. It is then not hard to see that $B=A_{\mathcal{V}_B,\mathcal{L}_B}$.
\end{proof}
We are now prepared to show the main theorem of this section:
\begin{theorem} \label{thm:accretive} Let $A=iS$, where $S$ is a strictly positive symmetric operator. Then, all dissipative extensions of $A$ are described by the operators $A_{\mathcal{V,L}}$, where $\mathcal{V}$ and $\mathcal{L}$ are as in Definition \ref{def:extension} and satisfy the condition
\begin{equation}
\mbox{\emph{Re}}\langle v,S_F\mathcal{L}v\rangle\geq\frac{1}{4}\|S_K^{1/2}(\mathcal{L}v+v)\|^2
\end{equation}
for any $v\in\mathcal{V}$.
\end{theorem}
\begin{proof} From Theorem \ref{thm:operator} and Lemma \ref{lemma:wadjoint}, we know that the domain of any dissipative extension of $A$ has to be contained in $\mathcal{D}(S_K^{1/2})$ and from Lemma \ref{lemma:param}, we have that any such extension can be written in the form $A_{\mathcal{V,L}}$ as defined in Definition \ref{def:extension}. Now, using that for any $v\in\mathcal{V}\subset\mathcal{D}(S_K ^{1/2})$, we have that $A_{\mathcal{V,L}}v=iS_F\mathcal{L}v$ and $W_A^*v=-iS_F^{1/2}\mathcal{P}v$, we may rewrite Condition \eqref{eq:conditiondiss} as follows:
\begin{align}
\Imag\langle v,A_{\mathcal{V,L}}v\rangle&\geq\frac{1}{4}\|(V_A^{-1/2}A_{\mathcal{V,L}}-W_A^*)v\|^2\notag\\
\Leftrightarrow\qquad \Imag\langle v,iS_F\mathcal{L}v\rangle&\geq\frac{1}{4}\|i(S_F^{-1/2}S_F\mathcal{L}+S_F^{1/2}\mathcal{P})v\|^2\notag\\
\Leftrightarrow\qquad \Real\langle v,S_F\mathcal{L}v\rangle&\geq\frac{1}{4}\|S_F^{1/2}(\mathcal{L}v+\mathcal{P}v)\|^2\notag\\
\Leftrightarrow\qquad \Real\langle v,S_F\mathcal{L}v\rangle&\geq\frac{1}{4}\|S_K^{1/2}(\mathcal{L}v+v)\|^2\:,
\end{align}
where we have used \eqref{eq:kreinform} for the last line. This shows the theorem.
\end{proof}
\begin{example} \label{ex:schr} \normalfont
Let us apply this result to the operator $A=iS$ on $L^2(\R^+)$, where $S$ is given by
\begin{align}
S:\qquad\mathcal{D}(S)=H_0^2(\R^+)=\{f\in H^2(\R^+): f(0)=f'(0)=0\},\quad
f\mapsto -f''+Vf
\end{align}
and $V$ is the operator of multiplication by a strictly positive bounded potential, i.e. $0<\varepsilon\leq V(x)$ almost everywhere and $V\in L^\infty(\R^+)$. Now let, $\eta\in H^2(\R^+)$ satisfy 
\begin{align} \label{eq:eta}
 \eta''(x)&=V(x)\eta(x)\notag\\
 \eta(0)&=1\:,
 \end{align}
 which means that $\eta$ spans $\ker S^*$. Then, using that  
 \begin{align}
 S_F^{1/2}:\qquad\mathcal{D}(S_F^{1/2})=H^1_0(\R^+),\quad
 f\mapsto\|f'\|^2+\langle f,Vf\rangle
 \end{align}
 and the decomposition in \eqref{eq:kreindomain}, we get
 \begin{equation}
 \mathcal{D}(S_K^{1/2})=\mathcal{D}(S_F^{1/2})\dot{+}\ker(S^*)=H^1_0(\R^+)\dot{+}\spann\{\eta\}=H^1(\R^+)\:.
 \end{equation}
 A calculation now shows that for any $f\in H^1(\R^+)$, we obtain
 \begin{equation}
 \|S_K^{1/2}f\|^2=\|f'\|^2+\langle f,Vf\rangle+\eta'(0)|f(0)|^2\:.
 \end{equation}
 Now, since $S$ is limit-circle at zero and limit-point at infinity, we get 
 \begin{equation}
 \dim\ker(S^*\pm i)=1
 \end{equation}
 and since $S$ is strictly positive, this implies that 
 \begin{equation}
 \dim\ker(S^*\pm i)=\dim\ker(S^*+\idty)=\dim\ker(A^*-i)=1\:.
 \end{equation}
 Thus, by Lemma \ref{prop:dissdomain}, we have only to look for extensions $A_{\mathcal{V,L}}$ where $\dim\mathcal{V}=1$ and $\mathcal{V}\cap H^2_0(\R^+)=\{0\}$ in order to describe all maximally dissipative extensions of $A$. Let $v$ be such that span $\mathcal{V}=\spann\{v\}$ and define $\mathcal{L}v=:\ell$. Moreover, denote $A_{v,\ell}:=A_{\mathcal{V,L}}$, which means that $A_{v,\ell}$ is given by
 \begin{align}
 A_{v,\ell}:\qquad\mathcal{D}(A_{v,\ell})=H^2_0(\R^+)\dot{+}\spann\{v\},\quad
 f+\lambda v\mapsto Af+i\lambda(-\ell''+V\ell)\:,
 \end{align}
 where $f\in H^2_0(\R^+)$ and $\lambda\in \C$. Applying Theorem \ref{thm:accretive}, we then find that $A_{v,\ell}$ is maximally dissipative (and thus $(-iA_{v,\ell})$ maximally accretive) if and only if 
 \begin{itemize}
 \item $v\in H^1(\R^+)$, but $v\notin H^2_0(\R^+)$.
 \item $v$ and $\ell$ satisfy the following condition:
 \begin{equation}
 \Real\left(\overline{v(0)}\ell'(0)\right)-\frac{\eta'(0)}{4}|v(0)|^2\geq\frac{1}{4}\left(\|v'-\ell'\|^2+\int_0^\infty|v(x)-\ell(x)|^2V(x)\text{d}x\right)\:,
 \end{equation}
 \end{itemize}
 where $\eta$ was determined by the conditions given in \eqref{eq:eta}.
\end{example}
\section{Dissipative extensions of $i\frac{{d}}{{d}x}+i\frac{\gamma}{x}$ on $L^2(0,1)$ } \label{sec:5}

Consider the closed dissipative operator $A$ on the Hilbert space $\mathcal{H}=L^2(0,1)$ given by
\begin{align}
A:\quad\mathcal{D}(A)=H^1_0(0,1)=\{f\in H^1(0,1): f(0)=f(1)=0\},\quad (Af)(x)=if'(x)+i\frac{\gamma}{x}f(x)\:,
\end{align}
where $\gamma>0$. In what follows, we will give a full description of all maximally dissipative extensions of $A$. 

Firstly, note that it can be shown that $\ker(A^*-i)=\spann\{x^\gamma e^x\}$ and thus by Lemma \ref{prop:dissdomain}, all maximally dissipative extensions $B$ of $A$ have to satisfy $\dim(\mathcal{D}(B)/ \mathcal{D}(A))=1$. This means that for each maximally dissipative extension $B$ of $A$, there exists a $v\notin H^1_0(0,1)$ such that $\mathcal{D}(B)=\mathcal{D}(A)\dot{+}\spann\{v\}$.
Next, let the operator $A_0$ be the restriction of $A$ to the set of compactly supported smooth functions on $(0,1)$, i.e. $A_0:=A\upharpoonright_{\mathcal{C}_c^\infty(0,1)}$, where it can be shown that $\overline{A_0}=A$.

Now, since $\mathcal{C}_c^\infty(0,1)$ is a core for $A$, it follows that $\mathcal{C}_c^\infty(0,1)\dot{+}\spann\{v\}$ is a core for $B$. Thus, we apply Theorem \ref{thm:operator} to extensions $B_0$ of $A_0$, whose domain is of the form $\mathcal{D}(B_0)=\mathcal{D}(A_0)\dot{+}\spann\{v\}$, where $v\notin H^1_0(0,1)$. If $B_0$ is dissipative, then we get that $B:=\overline{B_0}$ is a maximally dissipative extension of $A$.
Now, by Theorem \ref{thm:operator}, $B_0$ is dissipative if and only if $v\in\mathcal{D}(W_{A_0}^*)$ and $v$ satisfies 
\begin{equation}
\Imag\langle v,B_0v\rangle\geq \frac{1}{4}\|(V_{A_0}^{-1/2}B_0-W_{A_0}^*)v\|^2\:.
\end{equation}

Let us now determine the imaginary part $V_{A_0}$. Since ${q}_{A_0}$ is given by
\begin{align}
q_{A_0}:\qquad\mathcal{D}({q}_{A_0})=\mathcal{C}_c^\infty(0,1),\quad
f\mapsto\Imag\langle f,A_0f\rangle=\gamma\int_0^1\frac{|f(x)|^2}{x}\text{d}x\:,
\end{align}
we see that $V_{A_0}$ is given by the selfadjoint maximal multiplication operator by the function $\frac{\gamma}{x}$. This implies that $V_{A_0}^{-1/2}$ is the bounded selfadjoint operator of multiplication by $\sqrt{\frac{x}{\gamma}}$.
In order to be able to apply Theorem \ref{thm:operator}, let us firstly determine $\mathcal{D}(W_{A_0}^*)$. Observe that $\mathcal{D}(W_{A_0})=\ran (V_{A_0}^{1/2}\upharpoonright_{\mathcal{D}(A_0)})=\mathcal{C}_c^\infty(0,1)$. This follows from the fact that $\mathcal{D}(A_0)=\mathcal{C}_c^\infty(0,1)$ and $V_{A_0}^{1/2}$ --- the operator of multiplication by $\sqrt{\frac{\gamma}{x}}$ --- is a bijection from $\mathcal{C}_c^\infty(0,1)$ to $\mathcal{C}_c^\infty(0,1)$.
We therefore get
\begin{align}
W_{A_0}:\qquad\mathcal{D}(W_{A_0})&=\mathcal{C}_c^\infty(0,1)\notag\\
(W_{A_0}f)(x)&=\left(i\frac{\text{d}}{\text{d}x}+i\frac{\gamma}{x}\right)\left(\sqrt{\frac{x}{\gamma}}f(x)\right)=\sqrt{\frac{x}{\gamma}}\left(if'(x)+i\frac{2\gamma+1}{2x}f(x)\right)\:.
\end{align}
Now, $v\in\mathcal{D}(W_{A_0}^*)$ means that the map
\begin{equation}
f\mapsto \int_0^1\overline{v(x)}\sqrt{\frac{x}{\gamma}}\left(if'(x)+i\frac{2\gamma+1}{2x}f(x)\right)\text{d}x
\end{equation}
 is a bounded linear functional on $\mathcal{C}_c^\infty(0,1)$. This implies that $v\in\mathcal{D}(W_{A_0}^*)$ if and only if $\left(v(x)\sqrt{\frac{x}{\gamma}}\right)\in\mathcal{D}(K^*)$, where $K$ is the operator given by
\begin{equation}
\mathcal{D}(K)=\mathcal{C}_c^\infty(0,1),\quad (Kf)(x)=if'(x)+i\frac{2\gamma+1}{2x}f(x)\:.
\end{equation}
Now, it can be shown that 
\begin{equation}
K^*:\quad\mathcal{D}(K^*)=H^1_0(0,1)\dot{+}\spann\left\{x^{\gamma+\frac{1}{2}}\right\},\quad
(K^*v)(x)=iv'(x)-i\frac{2\gamma+1}{2x}v(x)
\end{equation}
and thus, we get
\begin{align} \label{eq:w*}
W_{A_0}^*:\quad\mathcal{D}(W_{A_0}^*)&=\left\{v\in L^2(0,1): (\sqrt{x}v(x))\in H^1_0(0,1)\dot{+}\spann\{x^{\gamma+\frac{1}{2}}\}\right\}\notag\\
(W_{A_0}^*v)(x)&=\left(i\frac{\text{d}}{\text{d}x}-i\frac{2\gamma+1}{2x}\right)\left(\sqrt{\frac{x}{\gamma}}v(x)\right)=i\left(\sqrt{\frac{x}{\gamma}}v(x)\right)'-i\frac{2\gamma+1}{2\sqrt{\gamma x}}v(x)\:. 
\end{align}

Pick a $v\in\mathcal{D}(W_{A_0}^*)$ such that $v\notin H^1_0(0,1)$. Then, for any such $v$ and any $\ell\in L^2(0,1)$, we introduce the operators $A_{0,v,\ell}$ given by

\begin{align}
A_{0,v,\ell}:\quad\mathcal{D}(A_{0,v,\ell})=\mathcal{C}_c^\infty(0,1)\dot{+}\spann\{v\},\quad
f+\lambda v\mapsto A_0f+\lambda\ell\:,
\end{align}
where $f\in\mathcal{C}_c^\infty(0,1)$ and $\lambda\in\C$. Note that $A_{0,v,\ell}v=\ell$. Now, by Theorem \ref{thm:operator}, Condition \eqref{eq:conditiondiss}, $A_{0,v,\ell}$ is dissipative if and only if
\begin{align} 
\Imag\langle v,A_{0,v,\ell}v\rangle&\geq\frac{1}{4}\|(V_{A_0}^{-1/2}A_{0,v,\ell}-W_{A_0}^*)v\|^2\notag\\
\Leftrightarrow \qquad \Imag\langle v,\ell\rangle&\geq\frac{1}{4}\int_0^1\left|\sqrt{\frac{x}{\gamma}}\ell(x)-i\left(\sqrt{\frac{x}{\gamma}}v(x)\right)'+i\frac{2\gamma+1}{2\sqrt{\gamma x}}v(x)\right|^2\text{d}x\:.\label{eq:firstordercond}
\end{align}
As argued above, we then get that the closure $A_{v,\ell}:=\overline{A_{0,v,\ell}}$ is a maximally dissipative extension of $A$. To summarize, we have shown the following 
\begin{proposition} \label{prop:firstordermaxdiss}
All maximally dissipative extensions of $A$ are given by the operators $A_{v,\ell}$
\begin{align}
A_{v,\ell}:\quad \mathcal{D}(A_{v,\ell})=H^1_0(0,1)\dot{+}\spann\{v\},\quad f+\lambda v\mapsto Af+\lambda\ell\:, 
\end{align}
where $v\in\mathcal{D}(W_{A_0}^*)\setminus H_0^1(0,1)$ (given in \eqref{eq:w*}) and $\ell\in L^2(0,1)$ satisfy Condition \eqref{eq:firstordercond}. 
\end{proposition}
\vspace{1cm}
\noindent {\bf Acknowledgements:} The main part of the research presented in this paper was done during the author's PhD studies (cf.\ \cite[Chapter 9.3]{thesis}).  It is thus a pleasure to thank Sergey Naboko and Ian Wood for
support and guidance as well as the UK Engineering
and Physical Sciences Research Council (Doctoral Training Grant Ref. EP/K50306X/1) and the School of
Mathematics, Statistics and Actuarial Science at the University of Kent for a PhD studentship.


\begin{thebibliography}{A}


\bibitem{Alonso-Simon} A.\ Alonso, B.\ Simon: The Birman-Kre\u\i n-Vishik theory of selfadjoint extensions of semibounded operators, J.\ Operator Theory {\bf{4}} (1980), 251-270.
\bibitem{Ando-Nishio} T.\ Ando, K.\ Nishio: Positive selfadjoint extensions of positive symmetric operators, Toh\'{o}ku Math. J., {\bf{22}} (1970), 65-75.


\bibitem{Arli} {Yu.\ Arlinski\u\i}: Boundary triplets and maximal accretive extensions of sectorial operators, Operator Methods for Boundary Value Problems. Ed. Seppo Hassi, Hendrik S. V. de Snoo, and Franciszek Hugon Szafraniec. 1st ed. Cambridge: Cambridge University Press, (2012), 35-72.

\bibitem{AT2009} {Yu.\ Arlinski\u\i} and E.\ Tsekanovski\u\i : M.\ Kre\u\i n's Research on Semi-Bounded Operators, its Contemporary Developments and Applications, Oper.\ Theory Adv.\ and Appl., {\bf{190}} (2009), 65-112.
\bibitem{AG82} Gr.\ Arsene and A.\ Gheondea: Completing matrix contractions, J. Operator Theory {\bf 7} (1982), 179-189.
\bibitem{Behrndtetal} J.\ Behrndt, S.\ Hassi, and H.\ de Snoo: Boundary Value Problems, Weyl Functions and Differential Operators, Monographs in Mathematics, vol.\ 108. Springer, Berlin (2020).

\bibitem{CrandallContraction} M.\ Crandall: Norm preserving extensions of linear transformations on Hilbert spaces, Proc.\ Amer.\ Math.\ Soc.\, {\bf 21} (1969), 335-340.
\bibitem{Crandall} M.\ Crandall and R.\ Phillips: On the extension problem for dissipative operators, Journ. Func. 
\bibitem{thesis} C.\ Fischbacher: On the Theory of Dissipative Extensions, PhD Thesis, University of Kent, 2017.
\bibitem{FNW} C.\ Fischbacher, S.\ Naboko and I. Wood: The Proper Dissipative Extensions of a Dual Pair, Integr. Equ. Oper. Theory 85 (2016), 573-599.
\bibitem{Nonproper} C.\ Fischbacher: The nonproper dissipative extensions of a dual pair, Trans. Am. Math. Soc. {\bf 370} (2018), 8895-8920. 
\bibitem{BKVG} C.\ Fischbacher: A Birman-Kre\u{\i}n-Vishik-Grubb Theory for Sectorial Operators, Complex Anal. Oper. Theory {\bf 13}, 3623-3658 (2019). 


\bibitem{Kato} T.\ Kato: Perturbation Theory for Linear Operators, Springer-Verlag, New York, 1966.

\bibitem{Teschl} G.\ Teschl: Mathematical Methods in Quantum Mechanics: With Applications to Schr\"odinger Operators, Graduate Studies in Mathematics, vol. 99, Amer. Math Soc., Providence, RI, 2009.
\bibitem{Phillips} R.\ Phillips: Dissipative operators and hyperbolic systems of partial differential equations, Trans.\ Amer.\ Math.\ Soc.\ {\bf 90} (1959), 192-254.


\end{thebibliography}
\end{document}